\documentclass{amsart}

\usepackage[latin1]{inputenc}
\usepackage{amsfonts}
\usepackage{amsmath}
\usepackage{amsthm}
\usepackage{amssymb}
\usepackage{latexsym}
\usepackage{enumerate}

\newtheorem{teo}{Theorem}[section]
\newtheorem{lema}[teo]{Lemma}
\newtheorem{prop}[teo]{Proposition}
\newtheorem{coro}[teo]{Corollary}
\newtheorem{fato}[teo]{Fact}
\newtheorem{remark}[teo]{Remark}

\theoremstyle{definition}
\newtheorem{defi}[teo]{Definition}

\numberwithin{equation}{section}

\begin{document}

\newcommand{\cc}{\mathfrak{c}}
\newcommand{\N}{\mathbb{N}}
\newcommand{\PP}{\mathbb{P}}
\newcommand{\forces}{\Vdash}

\title{Thin-very tall compact scattered spaces which are hereditarily separable}
\author{Christina Brech}
\thanks{The research has been a part of Thematic Project FAPESP (2006/02378-7). 
The first author was supported by scholarships from CAPES 
(3804/05-4) and CNPq (140426/2004-3 and 202532/2006-2). She would like to thank Stevo Todorcevic and
the second author, her Ph.D. advisors at the University of S\~ao
Paulo and at the University of Paris 7, under whose supervision the results of this paper were obtained.} 
\address{IMECC-UNICAMP, Caixa Postal 6065, 13083-970, Campinas, SP, Brazil}
\email{christina.brech@gmail.com}

\author{Piotr Koszmider}
\thanks{The second author was partially supported by Polish Ministry of Science and
 Higher Education research grant N N201 386234.}
\address{Instytut Matematyki; Politechnika {\L}\'odzka, ul. W\'olcza\'nska 215; 90-924  {\L}\'od\'z, Poland}
\email{pkoszmider.politechnika@gmail.com}
%
\subjclass{Primary 54G12; Secondary 03E35, 46B26}
%
%
%
\begin{abstract}
We strengthen the property $\Delta$ of a function $f:[\omega_2]^2\rightarrow [\omega_2]^{\leq \omega}$
considered by Baumgartner and Shelah. This allows us to consider new types of amalgamations
in the forcing used by Rabus, Juh\'asz and Soukup to construct thin-very tall compact scattered spaces.
We consistently obtain  spaces $K$ as above where $K^n$ is hereditarily separable for each $n\in\N$.
This serves as a counterexample concerning cardinal functions on compact spaces as well
as having some applications in Banach spaces: the Banach space $C(K)$ is an Asplund
space of density $\aleph_2$ which has no Fr\'echet smooth renorming, nor an
uncountable biorthogonal system.
\end{abstract}

\maketitle

\markright{THIN-VERY TALL COMPACT SCATTERED SPACES WHICH ARE HS}

\section{Introduction}

Given a compact scattered space $K$, we call the derivative of $K$ (denoted by
$K'$) the subset of $K$ formed by its accumulation points and  we inductively
define $K^{(\alpha)} = (K^{(\beta)})'$ if $\alpha = \beta +1$ and
$K^{(\alpha)} = \bigcap_{\beta<\alpha} K^{(\beta)}$ if $\alpha$ is a limit
ordinal. The height of $K$, $ht(K)$, is the smallest ordinal $\alpha$ such
that $K^{(\alpha)}$ is finite and nonempty, and the width of $K$, $wd(K)$, is
the supremum of the cardinalities $|K^{(\alpha)} \setminus K^{(\alpha+1)}|$
for $\alpha < ht(K)$. We call $K = \bigcup_{\alpha<ht(K)} K^{(\alpha)} 
\setminus K^{(\alpha+1)}$ the Cantor-Bendixson decomposition of $K$ and  $K^{(\alpha)} 
\setminus K^{(\alpha+1)}$ its $\alpha^{th}$ Cantor-Bendixson level.

The purpose of this work is to show that the existence of compact
hereditarily separable scattered spaces of height $\omega_2$ is consistent
with the usual axioms of set theory. For a given ordinal $\theta$ let us consider the
following notation:
\begin{itemize}
\item A cw($\theta$) space is a compact scattered space of countable width
   and height equal to $\theta$.
\item A hs($\theta$) space is a compact scattered space which is hereditarily
  separable and of height equal to $\theta$.
\end{itemize}
cw($\omega_1$) spaces are usually called thin-tall spaces and cw($\omega_2$) spaces are the thin-very tall spaces.
First we remark that any hs($\theta$) space is a cw($\theta$) space as the
Cantor-Bendixson levels form discrete subspaces. Whether
there is or not in ZFC a cw($\omega_1$) space was a
question posed by Telg\'arsky in 1968 (unpublished) and first (consistently) 
answered by Ostaszewski \cite{Ostaszewski}, using $\diamondsuit$. Rajagopalan 
constructed the first ZFC example of a cw($\omega_1$) in \cite{Rajagopalan}.
 Further, Juh\'asz and Weiss generalized these results (and simplified their proofs) in \cite{JuhaszWeiss}
proving in ZFC that for any ordinal $\theta<\omega_2$, there is a cw($\theta$)
space. 

For higher $\theta$'s the situation changes: in any  model of CH there are
no cw($\omega_2$) spaces and Just proved in \cite{Just} that neither are there such spaces in the Cohen
model (where $\neg$ CH holds). On the other
hand, Baumgartner and Shelah \cite{BaumgartnerShelah} constructed by forcing
the first consistent example of a cw($\omega_2$) space. An interesting point of
this forcing construction was the use of a new combinatorial device called
a function with the property $\Delta$.

The main purpose of this work is to prove the consistency of the existence of
a hs($\omega_2$) space. In fact, our space has even stronger properties: 
each of its finite powers is hereditarily separable. Whether
consistently there are hs($\omega_3$) or even cw($\omega_3$) spaces remains a well-known open question. On the other hand,
Martínez in \cite{MartinezIJM} adopted the method of \cite{BaumgartnerShelah} 
to obtain the consistency of the existence of cw($\theta$) spaces
for each $\theta<\omega_3$.

It follows from an old result of Shapirovski\u{\i} \cite{Shapirovskii} that for
any compact space $K$, $hd(K) \leq hL(K)^+$. Our construction shows that the dual
inequality does not follow from ZFC, since for our compact space $K$, we have
that $hL(K) = \aleph_2 \not\leq \aleph_1 = hd(K)^+$. Nevertheless, the dual 
inequality holds under GCH for regular spaces: since the weight $w(K)$ of a regular 
space $K$ is less or equal to $2^{d(K)}$ (see, for example, \cite{HodelHBST}), we 
trivially conclude that $hL(K) \leq 
w(K) \leq 2^{d(K)} = d(K)^+ \leq hd(K)^+$ under GCH. 

Turning to properties of Banach spaces, let us first recall some definitions
and results: a Banach space $X$ is an Asplund space if every continuous and
convex real-valued function on $X$ is Fr\'echet smooth at all points of
a G$_\delta$ dense subset of $X$. For separable Banach spaces, this is
equivalent to admitting a Fr\'echet smooth renorming (see \cite{DevilleGodefroyZizler}).
Namioka and Phelps proved in \cite{NamiokaPhelps} that $C(K)$ is
Asplund if and only if $K$ is scattered. Thus, our $C(K)$ is an Asplund
space. 

Haydon constructed in \cite{HaydonCSQ} the first nonseparable Asplund
space $C(K)$ which does not admit a Fr\'echet smooth renorming, concluding that the situation
changes for nonseparable Asplund spaces. Later, Jim\'enez Sevilla and Moreno
analyzed in \cite{JimenezMoreno} the structural properties of the space
$C(K)$, where $K$ is the well-known Kunen line constructed under CH (see
\cite{NegrepontisHBST}). 
They showed, for the Kunen line $K$, that $C(K)$ is also a nonseparable Asplund space
with no Fr\'echet smooth renorming. 

The weight of our space $K$ is $\aleph_2$, so that $C(K)$ is an Asplund space
of density $\aleph_2$. The fact that  $K$ is compact scattered and every
finite power of $K$ is hereditarily separable implies, in the same way
as for the Kunen line, that $C(K)$ does not
admit any Fr\'echet smooth renorming, but as in the case of the Kunen line we do not know if it
admits a G\^ateaux smooth renorming, or a Fr\'echet smooth bump
function.

A biorthogonal system on a Banach space $X$ is a family
$(x_\alpha, \varphi_\alpha)_{\alpha < \kappa} \subseteq X \times X^*$ such that
$\varphi_\alpha(x_\beta) = \delta_{\alpha, \beta}$ and a semi-biorthogonal
system on a Banach space $X$ is a sequence $(x_\alpha, \varphi_\alpha)_{\alpha
  < \kappa} \subseteq X \times X^*$ such that $\varphi_\alpha(x_\beta) = 1$,
if $\alpha=\beta$, $\varphi_\alpha(x_\beta)= 0$, if $\alpha<\beta$ and
$\varphi_\alpha(x_\beta)\geq 0$ if $\beta < \alpha$. Todorcevic showed in
\cite{TodorcevicBiorthogonal} (Theorem 9 together with the results of \cite{BorweinVanderwerff}) 
the existence of uncountable
semi-biorthogonal systems in Banach spaces $C(K)$ of density strictly greater
than $\aleph_1$. On the other hand, the fact that our space $K$ is compact scattered and every
finite power of $K$ is hereditarily separable implies, in the same way
as for the Kunen line, that $C(K)$ does not admit an
uncountable biorthogonal system. It follows that Todorcevic's
result cannot be improved in ZFC by replacing the existence of uncountable
semi-biorthogonal systems by the existence of uncountable biorthogonal systems
in spaces $C(K)$ of large density. On the other hand it is proved in
\cite{TodorcevicBiorthogonal} that it is consistent that every nonseparable Banach space
has an uncountable  biorthogonal system, showing that the existence
of a Banach space like ours or Kunen's 
cannot be proved in ZFC.

Our construction is based on the Juh\'asz and Soukup \cite{JuhaszSoukup} interpretation  
of Rabus' work \cite{Rabus}, where he
modified the Baumgartner-Shelah forcing from \cite{BaumgartnerShelah} to
obtain a countably tight space which is initially $\omega_1$-compact and
noncompact, answering a question of Dow and van Douwen. 

This paper is organized as follows: we finish this section by reviewing the
method of Juh\'asz and Soukup and some related
results and definitions which we will need afterwards. In Section 2 we prove the
key lemma which enables us to prove our main result in a straightforward
way. This lemma introduces a new way of  amalgamating conditions in
forcings which add thin-very tall spaces. One can apply these amalgamations
in the generic construction if one strengthens the property $\Delta$
of a function involved in the forcing.
In Section 3, we introduce the strong property $\Delta$ and assuming the existence of a function which satisfies it,
we prove the main results and analyze their consequences in topological and functional analytic
terms. Section 4 is devoted to establishing the consistency of the
existence of a function with the strong property $\Delta$.
Section 4 is due to the second author and the remaining sections to
the first author.

The notation and terminology used are those of \cite{JuhaszSoukup}. Given a set $X$, $\wp(X)$ is the power set of $X$ and, given a cardinal $\kappa$, $[X]^\kappa$ (resp. $[X]^{\leq\kappa}$ and $[X]^{<\kappa}$) denotes the family of subsets of $X$ of cardinality equal to $\kappa$ (resp. less or equal to $\kappa$ and less than $\kappa$).

Let us start by recalling the definition of the property $\Delta$:

\begin{defi}[Baumgartner, Shelah, \cite{BaumgartnerShelah}, p.122]
A function $f:[\omega_2]^2\rightarrow[\omega_2]^{\leq\omega}$ has the property
 $\Delta$ if $f(\{\xi, \eta\}) \subseteq \min\{\xi, \eta\}$ for all $\{\xi, \eta\}
 \in [\omega_2]^2$ and for any uncountable family $\mathcal{A}$ of finite subsets
of $\omega_2$, there are distinct $a,b\in\mathcal{A}$ such that for any $\tau\in a\cap b$, any $\xi\in a\setminus b$ and any $\eta\in b\setminus a$ we have:
\begin{enumerate}[1)]
\item $a \cap b \cap \min\{\xi,\eta\}\subseteq f(\{\xi,\eta\})$;
\item $\tau<\xi\ \Rightarrow f(\{\tau,\eta\})\subseteq f(\{\xi,\eta\})$;
\item $\tau<\eta\ \Rightarrow f(\{\tau,\xi\})\subseteq f(\{\xi,\eta\})$.
\end{enumerate}
\end{defi}

Now, we fix a function $f: [\omega_2]^2 \rightarrow [\omega_2]^{\leq \omega}$ with the property $\Delta$.

\begin{defi}[Juh\'asz, Soukup \cite{JuhaszSoukup}, Definition 2.1]\label{defiForcing}
Let $\PP_f$ be the forcing formed by conditions $p=(D_p, h_p, i_p)$ where:
        \begin{enumerate}[1.] 
                \item $D_p \in [\omega_2]^{< \omega}$;
                \item $h_p: D_p \rightarrow \wp(D_p)$ and for all $\xi \in
                D_p$, $\max h_p(\xi) = \xi$;
                \item $i_p: [D_p]^2 \rightarrow [D_p]^{< \omega}$ and for
                all $\xi,\eta \in D_p$, $\xi < \eta$, we have that:
                        \subitem{(a)} if $\xi \in h_p(\eta)$, then
                $h_p(\xi)\setminus h_p(\eta) \subseteq \bigcup_{\gamma \in
                i_p(\{\xi,\eta\})} h_p(\gamma)$,
                        \subitem{(b)} if $\xi \notin h_p(\eta)$, then
                $h_p(\xi) \cap h_p(\eta) \subseteq \bigcup_{\gamma \in
                i_p(\{\xi, \eta\})} h_p(\gamma)$,
                        \subitem{(c)} $i_p(\{\xi, \eta\}) \subseteq
                f(\{\xi,\eta\})$;
        \end{enumerate}
ordered by $p \leq q$ if $D_p \supseteq D_q$, for all $\xi \in D_q$, $h_p(\xi)
\cap D_q = h_q(\xi)$ and $i_p|_{[D_q]^2}= i_q$.
\end{defi}

To simplify notation, it is convenient to define the following:

\begin{defi}[Juh\'asz, Soukup \cite{JuhaszSoukup}]
Given finite nonempty sets of ordinals $x$ and $y$ such that $\max x < \max y$, we define
$$x*y = \left\{
                     \begin{array}{ll}
                      x \setminus y & \text{if }\max x \in y ,\\
                      x \cap y & \text{if } \max x \notin y.
                     \end{array} \right.$$
\end{defi}

We now rewrite conditions 3.(a) and (b) of the definition of the forcing as
$$h_p(\xi)* h_p(\eta) \subseteq \bigcup_{\gamma \in i_p(\{\xi,\eta\})} h_p(\gamma).$$

To define the space $K_f$, fix the ground model $V$ and a generic filter $G$.

\begin{defi}[Juh\'asz, Soukup \cite{JuhaszSoukup}, Definition 2.3]\label{defiEspacoGeral}
For each $\xi <\eta < \omega_2$, working in $V^{\PP_f}$, let
        $$h(\xi) = \bigcup_{p \in G} h_p(\xi) \quad \text{ and } \quad
        i(\{\xi,\eta\}) = \bigcup_{p \in G} i_p(\{\xi,\eta\}),$$ 
and let $L_f$ be the topological space $(\omega_2, \tau)$, where $\tau$ is the
        topology on $\omega_2$ which has the family of sets
        $$\{h(\xi): \xi < \omega_2\} \cup \{\omega_2 \setminus h(\xi): \xi <
        \omega_2\}$$
as a topological subbasis. We call $h(\xi)$ the generic neighborhood of $\xi$.
\end{defi}

From Theorem 1.5 of \cite{JuhaszSoukup}, it follows that for all $\xi<
\omega_2$, $h(\xi)$ is a compact subspace of $(\omega_2,\tau)$ and it easy to check that
        $$\{h(\xi)\setminus \bigcup_{\eta \in F} h(\eta) : F \in [\xi]^{<\omega}\} \text{ forms a local topological basis at $\xi$.} \eqno{(+)}$$
Therefore $L_f$ is a locally compact
scattered zero-dimensional space.

We are now ready to define $K_f$:

\begin{defi}\label{defiEspacoGenerico}
In $V^{\PP_f}$, $K_f$ is the one-point compactification of $L_f$. The 
point of compactification is denoted $*$, thus $K_f \setminus L_f=\{*\}$.
\end{defi}

In particular, we use the following results.

\begin{teo}[Rabus \cite{Rabus}, Lemma 4.1; Juh\'asz, Soukup
  \cite{JuhaszSoukup}, Lemma 2.8]\label{teoForcingCCC}
$\PP_f$ satisfies c.c.c.
\end{teo}

\begin{prop}\label{propLocalmenteCompacto}
$V^{\PP_f}$ satisfies ``$K_f$ is a compact scattered zero-dimensional space''.
\end{prop}

\section{Amalgamating conditions}

In this section, we present the key lemma needed to prove our main result. Let us
start with some preliminaries and auxiliary lemmas.

\begin{defi}\label{defiIsomorphicConditions}
Let $p_1 = (D_1, h_1, i_1)$, $p_2 =(D_2, h_2, i_2) \in \PP_f$ be two
conditions. We say that $p_1$ and $p_2$ are isomorphic conditions if there is
an order-preserving bijective function $e: D_1 \rightarrow D_2$ satisfying the
following conditions:
\begin{enumerate}[(a)]
         \item if $\xi, \eta \in D_1$, then $\xi \in h_1(\eta)$ if and only if
         $e(\xi) \in h_2(e(\eta))$;
         
         \item if $\xi \in D_1 \cap D_2$, then $e(\xi)= \xi$.
         
\end{enumerate}
In this case, if the order-preserving bijection $e$ is such that $\xi \leq e(\xi)$ for every $\xi \in D_1$ we say that $p_1$ is lower than $p_2$. 
\end{defi}

For example we have the following:

\begin{lema}\label{lemaAboutBijection} Let $p_1 = (D_1, h_1, i_1)$, $p_2 =(D_2, h_2, i_2) \in \PP_f$ be two
isomorphic conditions and let $e: D_1 \rightarrow D_2$ be the order-preserving bijection. 
Then for every $\xi \in D_1 \cap D_2$,
\begin{enumerate}[(a)]
      \item  $(h_1(\xi) \cup h_2(\xi))\cap D_1 = h_1(\xi)$, 
        \item $(h_1(\xi) \cup h_2(\xi))\cap D_2 = h_2(\xi)$,
\item $h_1(\xi) = e^{-1}[h_2(\xi)]$.
\end{enumerate}
\end{lema}
\begin{proof} 
Directly from Definition \ref{defiIsomorphicConditions}.
\end{proof}

\begin{defi}[Juh\'asz, Soukup \cite{JuhaszSoukup}]\label{defiDelta}
Given $p_1 = (D_1, h_1, i_1), p_2 =(D_2, h_2, i_2) \in \PP_f$, define a mapping $\delta_2: dom(\delta_2) \rightarrow D_1 \cap D_2$, where
$$dom(\delta_2)=\{\eta \in D_2: \text{there is } \delta \in D_1\cap D_2 \text{ such that } \eta \in h_2(\delta)\}$$
and
$$\delta_2(\eta) = \min \{\delta \in D_1 \cap D_2: \eta \in h_2(\delta)\}.$$
\end{defi}

\begin{lema}\label{lemaDelta} Suppose that $p_1 = (D_1, h_1, i_1)$ and  $p_2 =(D_2, h_2, i_2) \in \PP_f$ are two
conditions.
Then, 
\begin{enumerate}[(a)]
\item for all $\eta \in dom(\delta_2)\setminus D_1$, we have that $\eta < \delta_2(\eta)$ and
\item for all $\eta\in D_1\cap D_2$ we have $\eta\in dom(\delta_2)$ and $\delta_2(\eta)=\eta$.
\end{enumerate}
\end{lema}
\begin{proof}
Directly from Definition \ref{defiDelta}.
\end{proof}

We prove the next lemma, for the reader's convenience. 

\begin{lema}[Juh\'asz, Soukup \cite{JuhaszSoukup}]\label{lemaJuhasz}
Let $p_1 = (D_1, h_1, i_1)$, $p_2 =(D_2, h_2, i_2) \in \PP_f$ be two
isomorphic conditions.
If $\xi\in D_1\cap D_2$, then
        $$h_2(\xi)  = \delta_2^{-1}[h_1(\xi)] .$$
\end{lema}
\begin{proof}
Let $\xi \in D_1 \cap D_2$.

Suppose that $\eta\in dom(\delta_2)$ and 
$\delta_2(\eta)
 \in h_1(\xi)$. Since $\delta_2(\eta), \xi \in D_1 \cap D_2$, it follows
 from \ref{defiIsomorphicConditions}.(a) and (b) that $\delta_2(\eta)
 \in h_2(\xi)$. Suppose that $\eta \notin h_2(\xi)$. Then,
 $\eta \in h_2(\delta_2(\eta))*h_2(\xi)$ so that there is
 $\delta \in i_2(\{\delta_2(\eta), \xi\})$ such that $\eta 
\in h_2(\delta)$, which contradicts the minimality of 
$\delta_2(\eta)$ and concludes the proof of the inclusion $\delta_2^{-1}[h_1(\xi)]\subseteq h_2(\xi)$.

Reciprocally, if $\eta \in h_2(\xi)$,
 then $\eta \in dom(\delta_2)$. Suppose that
 $\delta_2(\eta) \notin h_2(\xi)$. 
Then, $\eta \in h_2(\delta_2(\eta))
 * h_2(\xi)$ so that there is 
$\delta \in i_2(\{\delta_2(\eta), \xi\})$
 such that $\eta \in h_2(\delta)$, which
 contradicts the minimality of $\delta_2(\eta)$. 
So, $\delta_2(\eta) \in h_2(\xi)$ and since $\delta_2(\eta),
 \xi \in D_1 \cap D_2$, it follows from \ref{defiIsomorphicConditions}.(a) 
and (b) that $\delta_2(\eta) \in h_1(\xi)$, concluding the proof of the lemma.
\end{proof}

In the proof of c.c.c., Rabus, Juh\'asz and Soukup considered the
minimal amalgamation which is constructed in
a symmetric way with respect to both of the conditions being extended. 
We will consider an asymmetric amalgamation. The lack of symmetry in our amalgamation is 
the result of using two functions, $\delta_2$ and $e$, in the definition
of the amalgamation.
The final auxiliary
lemma below characterizes the sets given by the operation $*$ for
elements of the extended condition. The role of the function $g$ will be played
by $\delta_2$ or by $e$.

\begin{lema}\label{lemaFuncaoG}
Let $p = (D_p, h_p, i_p) \in \PP_f$ and let $D_q \in [\omega_2]^{<
  \omega}$, $h_q: D_q \rightarrow \wp(D_q)$ and $g: dom(g)
\rightarrow D_p$ be such that
\begin{enumerate}[(i)]
        \item $D_p \subseteq D_q$; $dom(g)\subseteq D_q$, 
        \item for all $\xi \in D_p\cap dom(g)$ we have $g(\xi)=\xi$,

        \item for all $\xi \in D_p$, $h_q(\xi)  = h_p(\xi)\cup g^{-1}[h_p(\xi)]$.
\end{enumerate}
Then, for all $\xi, \eta \in D_p$, $\xi<\eta$, we have that
        $$(h_q(\xi) * h_q(\eta)) \cap D_p = h_p(\xi) * h_p(\eta)$$
and
        $$(h_q(\xi) * h_q(\eta)) \cap (D_q \setminus D_p) = g^{-1}[h_p(\xi) * h_p(\eta)] \cap (D_q \setminus D_p).$$
\end{lema}
\begin{proof}
Since $\xi, \eta \in D_p$, $\xi<\eta$, by (ii) we have that $\xi\in h_p(\eta)$ if and
only if $\xi\in h_q(\eta)$, so (ii) obviously gives $(h_q(\xi) * h_q(\eta)) \cap D_p = h_p(\xi) * h_p(\eta)$.

Now suppose $\xi \in h_q(\eta)$, so $\xi \in
h_p(\eta)$ and so by (iii)
$$(h_q(\xi)\setminus h_q(\eta)) \cap (D_q \setminus D_p) = (h_q(\xi)
        \cap (D_q\setminus D_p)) \setminus (h_q(\eta)\cap (D_q \setminus
        D_p))$$
        $$= (g^{-1}[h_p(\xi)] \cap (D_q \setminus D_p)) \setminus (g^{-1}[h_p(\eta)] \cap (D_q \setminus D_p))
 = g^{-1}[h_p(\xi)\setminus
        h_p(\eta)] \cap (D_q \setminus D_p).$$

On the other hand, if $\xi \notin
h_q(\eta)$, then $\xi \notin h_p(\eta)$ and so by (iii)
        $$(h_q(\xi)\cap h_q(\eta)) \cap (D_q \setminus D_p) = (h_q(\xi) \cap
        (D_q\setminus D_p)) \cap (h_q(\eta)\cap (D_q \setminus D_p))$$
        $$= (g^{-1}[h_p(\xi)] \cap (D_q \setminus D_p)) \cap (g^{-1}[h_p(\eta)] \cap 
(D_q \setminus D_p)) = g^{-1}[h_p(\xi)\cap
        h_p(\eta)] \cap (D_q \setminus D_p),$$
concluding the proof of the lemma.
\end{proof}

Now we go to our key lemma: a strong hypothesis about the behaviour of the
function $f$ allows us to amalgamate two isomorphic conditions, one lower than the other, into a common extension $q$
in such a way that
$h(\xi)\cap D_q\subseteq h[e(\xi)]\cap D_q$ for $\xi$ in the domain of the lower of the two conditions.

\begin{lema}\label{lemaAmalgamacaoNova}
Let $p_1 = (D_1, h_1, i_1)$, $p_2 =(D_2, h_2, i_2) \in \PP_f$ be two
isomorphic conditions and suppose $p_1$ is lower than $p_2$. Let $e: D_1 \rightarrow D_2$ be the order-preserving
bijective function and assume that
\begin{enumerate}[(A)]
        \item if $\xi, \eta \in D_1 \cap D_2$ and $\xi \neq \eta$, then
        $i_1(\{\xi, \eta\}) = i_2(\{\xi, \eta\})$;
        \item for all $\zeta \in D_1 \cap D_2$, all $\xi \in D_1\setminus D_2$
        and all $\eta \in D_2 \setminus D_1$:
\begin{enumerate}[(i)]
       \item if $\zeta < \xi$, then $f(\{\zeta, \eta\}) \subseteq
       f(\{\xi,\eta\})$;
       \item $D_1 \cap \xi \cap \eta \subseteq f(\{\xi, \eta\})$.
\end{enumerate}
\end{enumerate}
Then there is $q \in \PP_f$, $q \leq p_1, p_2$, such that for all $\xi \in D_1$
and all $\eta \in D_2$:
        $$\xi \in h_q(\eta) \text{ if and only if } e(\xi) \in h_2(\eta).$$
\end{lema}
\begin{proof} 
We define $q=(D_q, h_q, i_q)$ by: $D_q = D_1\cup D_2$;
        $$h_q(\xi) = \left\{
                     \begin{array}{ll}
                    
                     h_1(\xi) \cup \delta_2^{-1}[h_1(\xi)]  & \text{if
                     }\xi \in D_1,\\
                     h_2(\xi) \cup e^{-1}[h_2(\xi)] & \text{if } \xi \in D_2;
                     \end{array} \right.$$
and
$$i_q(\{\xi, \eta\}) = \left\{
                        \begin{array}{ll}
                        i_1(\{\xi, \eta\}) & \text{if } \xi, \eta \in D_1,\\
                        i_2(\{\xi, \eta\}) & \text{if } \xi, \eta \in D_2,\\
                        f(\{\xi, \eta\}) \cap D_q& \text{otherwise.}
                        \end{array} \right.$$
Note that (A) implies that the set $i_q(\{\xi, \eta\})$ is well-defined for any
$\xi, \eta \in D_1 \cap D_2$, $\xi \neq \eta$; clearly $i_q$ is
well-defined for the other pairs. Also, if $\xi\in D_1\cap D_2$, then the set
$h_q(\xi)$ is well-defined because both of the conditions reduce to
$h_q(\xi)=h_1(\xi)\cup h_2(\xi)$ by Lemmas \ref{lemaJuhasz} and \ref{lemaAboutBijection}.(c).

We have to show that $q \in \PP_f$, i.e., that $q$ satisfies conditions 1, 2
and 3 from Definition \ref{defiForcing}. The fact that $q$ satisfies
conditions \ref{defiForcing}.1 and \ref{defiForcing}.3.(c) follows directly
from the definition of $q$ and from the fact that $p_1, p_2 \in \PP_f$.
Condition \ref{defiForcing}.2. is satisfied because $p_1, p_2 \in \PP_f$ and
the functions $e$ and $\delta_2$ are nondecreasing. In what follows we will be using
Lemma \ref{lemaFuncaoG} for $p=p_1, p_2$ and $g=\delta_2, e$ respectively. 
The hypothesis of the lemma is satisfied for these objects by \ref{lemaDelta}.(b) and 
\ref{defiIsomorphicConditions}.(b).

Now we check conditions
  \ref{defiForcing}.3.(a) and (b).
Let $\xi, \eta \in D_q$, $\xi< \eta$, and
we consider the following cases:

\vspace{0.3cm}

\noindent \textit{Case 1. $\xi, \eta \in D_1$.}

It follows from the definition of $q$ and from Lemma \ref{lemaFuncaoG} that
         $$(h_q(\xi)*h_q(\eta)) \cap D_1 = h_1(\xi) * h_1(\eta)$$
and
         $$(h_q(\xi)*h_q(\eta)) \cap (D_2 \setminus D_1)= \delta_2^{-1}[h_1(\xi)*h_1(\eta)] \cap (D_2 \setminus D_1).$$

Now let $\zeta \in h_q(\xi)*h_q(\eta)$.

\vspace{0.3cm}

\noindent \textit{Subcase 1.1. $\zeta \in D_1$.}

In this subcase, $\zeta \in h_1(\xi)* h_1(\eta)$ and there is $\gamma \in
i_1(\{\xi, \eta\}) = i_q(\{\xi, \eta\})$ such that $\zeta \in h_1(\gamma)
\subseteq h_q(\gamma)$, as we wanted.

\vspace{0.3cm}

\noindent \textit{Subcase 1.2. $\zeta \in D_2 \setminus D_1$.}

In this subcase, $\delta_2(\zeta) \in h_1(\xi)* h_1(\eta)$ and, since
$\delta_2(\zeta), \xi, \eta\in D_1$ and $p_1 \in \PP_f$, there is $\gamma \in
i_1(\{\xi, \eta\}) = i_q(\{\xi, \eta\})$ such that $\delta_2(\zeta) \in
h_1(\gamma)$. Since $\gamma \in D_1$, it follows by the definition of $q$ that $\zeta \in
h_q(\gamma)$, as we wanted.

\vspace{0.3cm}

\noindent \textit{Case 2. $\xi, \eta \in D_2$.}

It follows from the definition of $q$ and from Lemma \ref{lemaFuncaoG} that
         $$(h_q(\xi)*h_q(\eta)) \cap D_1 = h_1(\xi) * h_1(\eta)$$
and
         $$(h_q(\xi)*h_q(\eta)) \cap (D_2 \setminus D_1)= e^{-1}[h_1(\xi)*h_1(\eta)] \cap (D_2 \setminus D_1).$$

Now let $\zeta \in h_q(\xi)*h_q(\eta)$.

\vspace{0.3cm}

\noindent \textit{Subcase 2.1. $\zeta \in D_1 \setminus D_2$.}

In this subcase, $e(\zeta) \in h_2(\xi)* h_2(\eta)$ and, since $e(\zeta), \xi,
\eta \in D_2$ and $p_2 \in \PP_f$, there is $\gamma \in i_2(\{\xi, \eta\}) =
i_q(\{\xi, \eta\})$ such that $e(\zeta) \in h_2(\gamma)$. Since $\gamma \in
D_2$, it follows by the definition of $q$ that $\zeta \in h_q(\gamma)$, as we wanted.

\vspace{0.3cm}

\noindent \textit{Subcase 2.2. $\zeta \in D_2$.}

In this subcase, $\zeta \in h_2(\xi)* h_2(\eta)$ and there is $\gamma \in
i_2(\{\xi, \eta\}) = i_q(\{\xi, \eta\})$ such that $\zeta \in h_2(\gamma)
\subseteq h_q(\gamma)$, as we wanted.

\vspace{0.3cm}

\noindent \textit{Case 3. $\xi \in D_1\setminus D_2$ and $\eta \in D_2 \setminus D_1$.}

Here we fix $\zeta \in h_q(\xi)*h_q(\eta)$ and we consider the following subcases:

\vspace{0.3cm}

\noindent \textit{Subcase 3.1. $\zeta \in D_1$.}

In this subcase, $\zeta \in D_1 \cap \xi \cap \eta$ and it follows from
(B).(ii) that $\zeta \in f(\{\xi, \eta\})$. Hence, $\zeta \in D_1 \cap
f(\{\xi, \eta\}) \subseteq D_q \cap f(\{\xi, \eta\})= i_q(\{\xi,
\eta\})$. Taking $\gamma = \zeta$, we conclude that $\zeta \in h_q(\gamma)$ and
$\gamma \in i_q(\{\xi, \eta\})$, as we wanted.

\vspace{0.3cm}

\noindent \textit{Subcase 3.2. $\zeta \in D_2 \setminus D_1$.}

First note that, regardless of the fact whether $h_q(\xi)*h_q(\eta)=h_q(\xi)\cap h_q(\eta)$ 
or $h_q(\xi)*h_q(\eta)=h_q(\xi)\setminus h_q(\eta)$, the assumption $\zeta \in h_q(\xi)*h_q(\eta)$ 
implies that $\zeta \in h_q(\xi)$ 

In this subcase,  it follows from the definition of $h_q(\xi)$ that $\delta_2(\zeta) \in
h_1(\xi)$, so that $\delta_2(\zeta) \in D_1 \cap \xi \cap \eta$; and it follows
from (B).(ii) that $\delta_2(\zeta) \in f(\{\xi, \eta\}) \cap D_q = i_q(\{\xi,
\eta\})$. By the definition of $\delta_2(\zeta)$, $\zeta \in h_2(\delta_2(\zeta))
\subseteq h_q(\delta_2(\zeta))$. Taking $\gamma = \delta_2(\zeta)$, we have that
$\zeta \in h_q(\gamma)$ and $\gamma \in i_q(\{\xi, \eta\})$, concluding the
proof of this subcase.

\vspace{0.3cm}

\noindent \textit{Case 4. $\xi \in D_2\setminus D_1$ and $\eta \in D_1 \setminus D_2$.}

Again we fix $\zeta \in h_q(\xi)*h_q(\eta)$ and we consider the following subcases:

\vspace{0.3cm}

\noindent \textit{Subcase 4.1. $\zeta \in D_1$.}

The proof in this subcase follows identically to the proof of Subcase 3.1.

\vspace{0.3cm}

\noindent \textit{Subcase 4.2\footnote{This case is similar to Subcase 2.2 in the proof
    of Claim 2.7.2 of \cite{JuhaszSoukup}.}. $\zeta \in D_2 \setminus D_1$.}

We start this last subcase by proving the following:

\vspace{0.3cm}

\noindent \textbf{Fact 1.} \textit{$\{\zeta, \xi\}\cap dom(\delta_2)$ is a
  nonempty set such that $\min\delta_2[\{\zeta,\xi\}]< \eta$ and if 
both
  $\zeta$ and $\xi$ are in $dom(\delta_2)$, then $\delta_2(\zeta) \neq \delta_2(\xi)$.}

\vspace{0.3cm}

\noindent \textit{Proof of Fact 1.} First remark that, from the definition of
  $*$, it follows that if $\xi \notin h_q(\eta)$, then $\zeta \in h_q(\xi) *
  h_q(\eta) = h_q(\xi) \cap h_q(\eta)$ and therefore $\zeta \in
  h_q(\eta)$. Analogously, if $\xi \in h_q(\eta)$, then $\zeta \in h_q(\xi) *
  h_q(\eta) = h_q(\xi) \setminus h_q(\eta)$ and therefore $\zeta \notin
  h_q(\eta)$. So,
        $$|\{\zeta, \xi\} \cap h_q(\eta)|=1.$$
From the definition of $q$ we have that, since $\xi,\zeta\not\in D_1$ and
$\eta\in D_1$ in this subcase, the above means that
        $$|\{\zeta, \xi\} \cap \delta_2^{-1}[h_1(\eta)]| = 1,$$
so that $\{\zeta, \xi\}\cap dom(\delta_2)$ is a nonempty set. The above observation
also implies that $\delta_2[\{\zeta,\xi\}]\cap h_1(\eta)\not=\emptyset$ and
so, $\min\delta_2[\{\zeta,\xi\}]\leq\eta$. Now  
since $\eta \notin D_2$ and the range of $\delta_2$ is included in $D_1\cap D_2$, the
inequality must be strict.

Finally, we have seen that $\zeta \in h_q(\eta)$ if and only if $\xi \notin
h_q(\eta)$ and so, if both $\zeta$ and $\xi$ are in the domain of $\delta_2$, it follows
that $\delta_2(\zeta) \in h_1(\eta)$ if and only if $\delta_2(\xi) \notin
h_1(\eta)$, so that $\delta_2(\zeta) \neq \delta_2(\xi)$, concluding the proof of
Fact 1.

\vspace{0.3cm}

Take $\theta = \min\{\delta_2(\xi), \delta_2(\zeta)\}$ and note that $\theta \neq
\xi$ since $\xi\in D_2\setminus D_1$ and the range of $\delta_2$ is
included in $D_1\cap D_2$. We go now to the following subcases:

\vspace{0.3cm}

\noindent \textit{Subcase 4.2.1. $\theta <\xi$.}

Here, $\theta \neq \delta_2(\xi)$ and therefore $\delta_2(\zeta) = \theta \in D_1
\cap \xi \cap \eta$. From condition (B).(ii), it follows that $\delta_2(\zeta)
\in f(\{\xi,\eta\}) \cap D_q = i_q(\{\xi, \eta\})$. Since $\zeta \in
h_2(\delta_2(\zeta)) \subseteq h_q(\delta_2(\zeta))$, taking $\gamma =
\delta_2(\zeta)$, we have that $\zeta \in h_q(\gamma)$ and $\gamma \in
i_q(\{\xi,\eta\})$, as we wanted.

\vspace{0.3cm}

\noindent \textit{Subcase 4.2.2. $\theta >\xi$.}

Note that  $\zeta \in h_q(\xi)*h_q(\eta) \subseteq h_q(\xi)$.  
Since $\zeta$ and $\xi$ satisfying the hypothesis of the Case 4.2. are in $D_2\setminus D_1$,
it follows from the definition of $h_q$ 
 that $\zeta \in h_2(\xi)$. To finish, let us show the
following:

\vspace{0.3cm}

\noindent \textbf{Fact 2.} \textit{$\zeta \in h_2(\xi)*h_2(\theta)$.}

\vspace{0.3cm}

\noindent \textit{Proof of Fact 2.} First suppose $\theta = \delta_2(\xi)$.
If $\zeta\not\in dom(\delta_2)$, then $\zeta \notin h_2(\delta_2(\xi))$. 
If $\zeta\in dom(\delta_2)$, 
 from Fact 1 and
the minimality of $\delta_2(\zeta)$, it follows that
$\zeta \notin h_2(\delta_2(\xi))$. Since $\xi \in h_2(\delta_2(\xi))$, we have that
$\zeta \in h_2(\xi) \setminus h_2(\delta_2(\xi)) = h_2(\xi)*h_2(\theta)$.

Now suppose $\theta = \delta_2(\zeta)$. Analogously we prove that $\xi \notin
h_2(\delta_2(\zeta))$ and $\zeta \in h_2(\xi) \cap h_2(\delta_2(\zeta)) =
h_2(\xi)*h_2(\theta)$, concluding the proof of Fact 2.

\vspace{0.3cm}

Finally, since $p_2 \in \PP_f$, there is $\gamma \in i_2(\{\xi, \theta\})$
such that $\zeta \in h_2(\gamma) \subseteq h_q(\gamma)$. By condition
(B).(i), which can be used by Fact 1, we have that
        $$i_2(\{\xi, \theta\}) \subseteq f(\{\xi,\theta\}) \cap D_2 \subseteq
        f(\{\xi,\eta\}) \cap D_q = i_q(\{\xi, \eta\}).$$
Hence, $\gamma \in i_q(\{\xi, \eta\})$ and $\zeta \in h_q(\gamma)$,
concluding the proof of Subcase 4.2, Case 4 and thus concluding the proof of
Claim 2.

\vspace{0.3cm}

Now we know that $q \in \PP_f$ and let us check the other conclusions: it
follows easily from the definition of $q$ and Lemma \ref{lemaJuhasz}
that $q \leq
p_1$ and analogously it follows from the definition of $q$ and Lemma \ref{lemaDelta}
that $q \leq p_2$.

Finally, we verify the condition we want $q$ to satisfy, that
is, $\xi\in h_2(\eta)\cup e^{-1}[h_2(\eta)]$ if and
only if $e(\xi)\in h_2(\eta)$: let $\xi \in D_1$ and
$\eta \in D_2$ and we consider again the following cases:

\vspace{0.3cm}

\noindent \textit{Case 1. $\xi \in D_1 \cap D_2$.}

It follows from the fact that in this case $e(\xi) =\xi$.

\vspace{0.3cm}

\noindent \textit{Case 2. $\xi \in D_1 \setminus D_2$.}

In this case, $\xi\in h_2(\eta)\cup e^{-1}[h_2(\eta)]$ if and
only if $\xi\in e^{-1}[h_2(\eta)]$ if and only if 
 $e(\xi) \in h_2(\eta)$, concluding the proof of the lemma.
\end{proof}

\section{The main results}

To apply the key lemma proved in the previous section, 
the function $f$ on which the forcing $\PP_f$ depends must
  satisfy a stronger version of the property $\Delta$:

\begin{defi}\label{strongDeltaProperty}
A function $f:[\omega_2]^2\rightarrow[\omega_2]^{\leq\omega}$ has the strong property
 $\Delta$ if $f(\{\xi, \eta\}) \subseteq \min\{\xi, \eta\}$ for all $\{\xi, \eta\}
 \in [\omega_2]^2$ and for any uncountable $\Delta$-system $\mathcal{A}$ of finite subsets
of $\omega_2$, there are distinct $a,b\in\mathcal{A}$ and an
order-preserving bijection $e:a\rightarrow b$ which is the identity on $a\cap b$
and such that $\xi\leq e(\xi)$ for all $\xi\in a $ and
for any $\tau\in a\cap b$, any $\xi\in a\setminus b$ and any $\eta\in b\setminus a$ we have:
\begin{enumerate}[1)]
\item $a\cap\min\{\xi,\eta\}\subseteq f(\{\xi,\eta\})$;
\item $\tau<\xi\ \Rightarrow f(\{\tau,\eta\})\subseteq f(\{\xi,\eta\})$;
\item $\tau<\eta\ \Rightarrow f(\{\tau,\xi\})\subseteq f(\{\xi,\eta\})$.
\end{enumerate}
\end{defi}

Finally we arrive at the main result of this paper. 

\begin{teo}\label{teoHeredDens2}
If $f: [\omega_2]^2 \rightarrow [\omega_2]^{\leq
  \omega}$ has the strong property $\Delta$, then  
  $V^{\PP_f}$ satisfies ``for all $n \in \mathbb{N}$, $K_f^n$ is hereditarily
separable''.
\end{teo}
\begin{proof}
We prove this by induction on $n \in \N$: in $V^{\PP_f}$, fix $n \in \N$ and
suppose that for all $0 \leq i <n$, $K_f^{i}$ is hereditarily separable (take
$K_f^0 = \{*\}$) and let us show that $K_f^n$ is hereditarily separable.
We will be using a well-known fact that a regular space is hereditarily separable
if and only if it has no uncountable left-separated sequence (see Theorem 3.1 of \cite{Roitman}).

In $V$, suppose $(\dot{x}_\alpha)_{\alpha<\omega_1}$ is a sequence of names
such that $\PP_f$ forces that $(\dot{x}_\alpha)_{\alpha< \omega_1}$ is a
left-separated sequence in $K_f^n$ of cardinality $\aleph_1$ and for each
$\alpha<\omega_1$, we have that $\dot{x}_\alpha = (\dot{x}^\alpha_1, \dots,
\dot{x}^\alpha_n)$, where each $\dot{x}_i^\alpha$ is a name for an element of
$K_f$.

Notice that if
       $$\PP_f \Vdash  \exists 1 \leq i \leq n, \ \exists X \subseteq
       \omega_1, \ |X|=\aleph_1 \text{ such that }\forall \alpha, \beta \in X,
       \ \dot{x}_i^\alpha = \dot{x}_i^\beta,$$
then
       $$\PP_f \Vdash \exists 1 \leq i \leq n, \ \exists X \subseteq \omega_1,
       \ |X|=\aleph_1 \text{ such that } ((\dot{x}_1^\alpha, \dots,
       \dot{x}_{i-1}^\alpha, \dot{x}_{i+1}^\alpha, \dots,
       \dot{x}_n^\alpha))_{\alpha  \in X}$$
       $$\text{ is a left-separated sequence in }K_f^{n-1},$$
contradicting the inductive hypothesis. Therefore, we can assume without loss
of generality that $\PP_f$ forces that for all $1 \leq i \leq n$ and all
$\alpha< \beta < \omega_1$, $\dot{x}_i^\alpha \neq \dot{x_i}^\beta$ and
$\dot{x}_i^\alpha \in L_f = K_f \setminus \{*\}$.

By assertion (+) following Definition \ref{defiEspacoGeral}, for each $\alpha<\omega_1$, there are names $\dot{F}^\alpha_1, \dots,
\dot{F}^\alpha_n$  for finite subsets of $\omega_2$ such that $\PP_f$
forces that
       $$\forall \alpha<\omega_1 \quad \forall 1\leq i \leq n \quad
       \dot{x}^\alpha_i \in h(\dot{x}^\alpha_i) \setminus \bigcup_{\xi\in
         \dot{F}^\alpha_i} h(\xi)$$
and
       $$\forall \alpha <\beta < \omega_1\quad \exists 1\leq i \leq n \quad
       \dot{x}^\alpha_i \notin h(\dot{x}^\beta_i) \setminus \bigcup_{\xi\in
       \dot{F}^\beta_i} h(\xi).$$

For each $\alpha<\omega_1$, let $p_\alpha=(D_\alpha, h_\alpha, i_\alpha)\in
\PP_f$, $x^\alpha_1, \dots, x^\alpha_n \in \omega_2$ and $F^\alpha_1, \dots,
F^\alpha_n \subseteq \omega_2$ be finite such that
       $$p_\alpha \Vdash \forall 1\leq i \leq n \quad \dot{x}^\alpha_i =
       \check{x}^\alpha_i    \text{ and } \dot{F}^\alpha_i =
       \check{F}^\alpha_i.$$

By Lemma 2.2 of \cite{JuhaszSoukup}, we can assume without loss of generality that
for all $\alpha<\omega_1$ and all $1\leq i\leq n$, $F^\alpha_i \subseteq
D_\alpha$ and $x^\alpha_i \in D_\alpha$.

By the $\Delta$-system Lemma, we can assume as well that
$(D_\alpha)_{\alpha<\omega_1}$ forms a $\Delta$-system with root $D$. Since for
each pair $\{\xi, \eta\} \subseteq D$ and each $\alpha < \omega_1$, we have
that $i_\alpha(\{\xi, \eta\}) \in [f(\{\xi, \eta\})]^{< \omega}$, we may
assume that for all $\alpha<\beta<\omega_1$, if $\xi, \eta \in D$, $\xi \neq
\eta$, then $i_\alpha(\{\xi, \eta\}) = i_\beta(\{\xi, \eta\})$.

By thinning out, we can assume without loss of generality that
$(D_\alpha)_{\alpha<\omega_1}$ forms a $\Delta$-system with root $D$ such that
for every $\alpha < \beta < \omega_1$:
\begin{itemize}
\item $p_\alpha$ is isomorphic to $p_\beta$;
\item $p_\alpha$ is lower than $p_\beta$;
\item if $e_{\alpha\beta}:D_\alpha \rightarrow D_\beta$ 
is the order-preserving bijective function, then $e_{\alpha\beta}(x_i^\alpha) = x_i^\beta$,
for all $1 \leq i \leq n$.
\end{itemize}

Finally, we may assume that for all $1 \leq i \leq n$ we have: either
$x^\alpha_i = x^\beta_i$ for all $\alpha<\beta<\omega_1$; or $x^\alpha_i
\notin D$ for all $\alpha<\omega_1$ and actually the second case holds by our 
initial assumption about the sequence.

Since $f$ has the strong property $\Delta$, there
are $\alpha<\beta<\omega_1$ such that for all $\zeta \in D$, all $\xi \in
D_\alpha\setminus D$ and all $\eta \in D_\beta \setminus D$:
\begin{enumerate}[(i)]
        \item $D_\alpha \cap \xi \cap \eta \subseteq f(\{\xi, \eta\})$;
        \item if $\zeta < \xi$, then $f(\{\zeta, \eta\}) \subseteq
        f(\{\xi,\eta\})$;
        \item if $\zeta < \eta$, then $f(\{\zeta, \xi\}) \subseteq
        f(\{\xi,\eta\})$.
\end{enumerate}

Note that $p_\alpha$ and $p_\beta$ satisfy the hypothesis of Lemma
\ref{lemaAmalgamacaoNova}. Hence, there is $q\leq p_\alpha, p_\beta$ in
$\PP_f$ such that for all $\xi \in D_\alpha$ and all $\eta \in D_\beta$,
        $$\xi \in h_q(\eta) \text{ if and only if } e_{\alpha\beta}(\xi) \in
        h_{p_\beta}(\eta).$$

Then, for all $1 \leq i \leq n$ and all $\xi \in D_\beta$, we have that
        $$x_i^\alpha \in h_q(\xi) \text{ if and only if } x_i^\beta \in
        h_{p_\beta}(\xi).$$
So we have that
        $$x_i^\alpha \in h_q(x_i^\beta) \setminus \bigcup_{\xi \in F_i^\beta}
        h_q(\xi).$$
But $q \leq p_\alpha, p_\beta$ and then
        $$q \Vdash \forall 1 \leq i \leq n, \ \dot{x}_i^\alpha =
        \check{x}_i^\alpha, \ \dot{x}_i^\beta = \check{x}_i^\beta \text{ and
        }\dot{F}_i^\beta = \check{F}_i^\beta.$$
Therefore, 
        $$q \Vdash \forall 1 \leq i \leq n, \ \dot{x}_i^\alpha =
        \check{x}_i^\alpha \in h(\check{x}_i^\beta) \setminus \bigcup_{\xi
          \in\check{F}_i^\beta} h(\xi) = h(\dot{x}_i^\beta) \setminus
        \bigcup_{\xi \in\dot{F}_i^\beta} h(\xi),$$
contradicting the hypothesis about $\dot{x}_i^\alpha, \dot{x}_i^\beta$ and $\dot{F}_i^\beta$.
\end{proof}

\begin{coro}
It is relatively consistent with ZFC that there is a hereditarily separable compact scattered
space of height $\omega_2$.
\end{coro}
\begin{proof}
Since each level of the Cantor-Bendixson decomposition of $K_f$ is a discrete subset of $K_f$, it follows that every level of it is countable. But $|K_f|=\aleph_2$ and $K_f = \bigcup_{\alpha<ht(K_f)} K_f^{(\alpha)} \setminus K_f^{(\alpha+1)}$, so that $ht(K_f) \geq \omega_2$. It is easy to see that $\bigcap_{\alpha<\omega_2} K_f^{(\alpha)} = \{*\}$ concluding that $ht(K_f) = \omega_2$.
\end{proof}

\begin{coro}
It is relatively consistent with ZFC that there is a hereditarily separable compact space with
hereditary Lindel\"of degree equal to $\aleph_2$. In particular, it is relatively consistent with ZFC that there is a compact space $K$ such that 
$hL(K) \not\leq hd(K)^+$''.
\end{coro}
\begin{proof}
It follows from the fact that $hL(K_f) \leq |K_f|=\aleph_2$ and that $\{K_f \setminus K_f^{(\alpha)}: \alpha< \omega_2\}$ is an open covering of $K_f \setminus\{*\}$ which does not admit a subcovering of strictly smaller cardinality.
\end{proof}

\begin{coro}
It is relatively consistent with ZFC that there is an Asplund space $C(K)$ of density $\aleph_2$
which does not admit any Fr\'echet smooth renorming and which does not contain an
uncountable biorthogonal system.
\end{coro}
\begin{proof}
Since every finite power of $K_f$ is hereditarily separable,
Lemma 4.37 and Theorem 4.38 of \cite{BookBiorthogonal} imply that $C(K_f)$ is hereditarily
Lindel\"of relative to its pointwise convergence topology. But for compact scattered spaces $K$, the pointwise convergence topology and the weak topology of $C(K)$ coincide (see Theorem 7.4 of \cite{NegrepontisHBST}), so that $C(K_f)$ is hereditarily Lindel\"of relative to its weak topology. 
 

Now, if $C(K_f)$ admits a Fr\'echet smooth renorming, by Corollaries 8.34 (due to Mazur \cite{Mazur}) and 8.36 of \cite{BookBiorthogonal} (due to Jim\'enez Sevilla and Moreno \cite{JimenezMoreno}) it contains an uncountable bounded subset $A$ such that for every $x_0 \in A$, $x_0$ is 
not in the (norm-) closed convex hull of $A \setminus\{x_0\}$, that is, $x_0 \notin \overline{conv}(A \setminus \{x_0\})$. Since the weak and norm convex closures coincide in Banach spaces, $A$ turns out to be an uncountable discrete family of $C(K_f)$ relative to its weak topology, which contradicts the fact that $C(K_f)$ is hereditarily Lindel\"of relative to its weak topology.

Now, if $C(K_f)$ admits an uncountable biorthogonal system 
$(x_\alpha, \varphi_\alpha)_{\alpha<\omega_1} \subseteq C(K_f)
 \times C(K_f)^*$, then $\{x_\alpha: \alpha<\omega_1\}$ is an
 uncountable discrete family of $C(K_f)$ relative to its
 weak topology, contradicting the fact that $C(K_f)$ is hereditarily
 Lindel\"of relative to its weak topology.
\end{proof}

One should compare the above corollary to Theorem 4.41 of \cite{BookBiorthogonal} (due to Ostaszewski \cite{Ostaszewski}) and to Corollary 8.37 of \cite{BookBiorthogonal} (due to Jim\'enez Sevilla and Moreno \cite{JimenezMoreno}). 

\section{The existence of the required function $f$}

In this section we prove the consistency of the existence of
a function with the strong property $\Delta$. It turns out that we are even able to prove
the consistency of the existence of such a function with its range included in the
family of finite (rather than countable) subsets of
$\omega_2$. The method is quite involved but, as shown at the end of this section,
forcings preserving CH (as in \cite{BaumgartnerShelah}) cannot serve for this purpose 
even if we were interested in a function with its range included in countable subsets of $\omega_2$.

\subsection{Forcing with side conditions in Velleman's simplified morasses.}

To construct a forcing which adds the required auxiliary function on pairs of $\omega_2$ we will
need a family of countable subsets of $\omega_2$ with some strong properties. The
following proposition establishes a list of the most useful properties:

\begin{prop}\label{prop1}
It is relatively consistent with ZFC+CH that there exists
a family $\mathcal{F}\subseteq [\omega_2]^{\omega}$ which satisfies the following properties:
\begin{enumerate}[1)]
\item $(\mathcal{F},\subseteq)$ is well-founded (thus, one can talk about $rank(X)$ for $X\in\mathcal{F}$);
\item $\mathcal{F}$ is stationary in $[\omega_2]^{\omega}$ (see \cite{BaumgartnerPFA});
\item If $\alpha\in X,Y\in \mathcal{F}$ and $rank(X)\leq rank(Y)$, then $X\cap\alpha\subseteq Y\cap \alpha$;
\end{enumerate}
 If $M$ is a countable elementary submodel of $H(\omega_3)$ containing $\omega_1,\omega_2, \mathcal{F}$ and $X=M\cap\omega_2\in \mathcal{F}$, then
\begin{enumerate}[1)]\setcounter{enumi}{3}
\item $M\cap\omega_1=rank(X)$;
\item $Y\subset X$, $Y\in \mathcal{F}$ implies $Y\in M$;
\item $X_1,..., X_n\in \mathcal{F}$ for $n\in N$ and $rank(X_i)<rank(X)$ for $1\leq i\leq n$ implies
that there is $Z\in \mathcal{F}$ such that $Z\in M$ and $X\cap(X_1\cup...\cup X_n)\subseteq Z$.
\end{enumerate}
\end{prop}
\begin{proof}
We will prove that a simplified Velleman's $(\omega_1,1)$-morass (see \cite{Velleman2})
which is a stationary coding set (see \cite{Zwicker})
satisfies the above properties. The proof relies heavily on the properties of Velleman's
morasses obtained in \cite{KoszmiderSCUF}. We will often refer to this paper, in particular we adopt 
definitions of simplified morass and stationary coding set from this paper (section 2).
The consistency of the existence of such morasses can be immediately
obtained from the corresponding proof for semimorasses in \cite{KoszmiderSemimorasses}, Theorem 3 Section 2.
 
1) follows from Definition 2.1 of \cite{KoszmiderSCUF} and 2) from the fact that
$\mathcal{F}$ is assumed to be a stationary coding set. To prove 3) apply 2.5 of \cite{KoszmiderSCUF}.
Now 4) is Fact 2.7 of \cite{KoszmiderSCUF},
5) is Fact 2.6 of \cite{KoszmiderSCUF} To obtain 6) apply Fact 2.8 of \cite{KoszmiderSCUF} to each
$X_i$ obtaining $Z(X_i)$ such that $Z(X_i)\in M\cap \mathcal{F}$ and $X_i\cap X\subseteq Z(X_i)$.
Now use the elementarity of $M$ and the directedness of $\mathcal{F}$ (see Definition 2.1. of \cite{KoszmiderSCUF})
to obtain $Z$ as in 6). 
\end{proof}

Now we will adopt a few facts from \cite{KoszmiderUMSUF} and \cite{KoszmiderSCUF} concerning
forcing with side conditions in $\mathcal{F}$. As explained in these papers,
to use  elements of $\mathcal{F}$ as side conditions means to use forcings $P$ whose 
conditions are of the form $(p, A)$ where $p$ is a finite condition of a natural forcing 
adding the structure in question and $A$ is a finite subset of $\mathcal{F}$. This
is like using models as side conditions in the method of forcing with models as side 
conditions developed by Todorcevic (see \cite{TodorcevicProper}). The order is given by the forcing order on the first 
coordinate and inverse inclusion on the second coordinate. In addition we require the
existence of some natural projections of $p$ onto the elements of $A$ as a part of the 
definition of the forcing notion. The properties 1) - 6) above allow us to perform many 
maneuvers with ease; also the definitions are simpler. This method appears to be equivalent 
to the variant of Todorcevic's method where one employs matrices of models (see 
\cite{TodorcevicDirectedSets} Section 4, for an example with detailed definitions).
The price we need to pay for this convenience is that $P$ is not proper (unlike 
Todorcevic's forcings,) but only $\mathcal{F}$-proper, i.e., there is a club 
$\mathcal{C} \subseteq[\omega_2]^{\omega}$ such that for models
$M\prec H(\omega_3)$ such that $M\in \mathcal{F}\cap \mathcal{C}$
and $p\in P\cap M$, there are $(P,M)$-generic conditions stronger than
$p$. As $\mathcal{F}$ may be assumed to be stationary, $\mathcal{F}$-properness implies the 
preservation of $\omega_1$ (proof as for proper forcings, see \cite{BaumgartnerPFA}). The 
preservation of bigger cardinals follows from the $\omega_2$-chain condition. Note
that the fact that the forcing is not proper but
preserves cardinals is no limitation in the applications that one seeks here, i.e.,
consistent existence of structures of sizes bigger than $\omega_1$.
Let us describe basic notions related to  forcing with side conditions in $\mathcal{F}$ that 
we will use.

\begin{defi}\label{defiFproper} 
Suppose $\mathcal{F}\subseteq[\omega_2]^\omega$.
We say that a forcing notion $P$ is $\mathcal{F}$-proper if there is $\theta>(2^{|P|})^+$ and a club set 
$\mathcal{C}\subseteq[H(\theta)]^\omega$  such that
whenever $p\in M\in \mathcal{C}$ and $M\cap\omega_2\in\mathcal{F}$ then there is
a $(P,M)$-generic $p_0\leq p$, i.e., 
$D\cap M$ is predense below $p_0$ for every $D\in M$ which is dense in $P$.
\end{defi}

\begin{fato}\label{fato3}
Suppose $\mathcal{F}\subseteq[\omega_2]^\omega$
is a stationary set and $P$ is an $\mathcal{F}$-proper forcing notion, then $P$ preserves 
$\omega_1$.
\end{fato}
\begin{proof}
The proof is a straightforward version of Shelah's paradigmatic
proof of preservation of $\omega_1$ by proper forcings (see 
\cite{ShelahProperForcing} or \cite{BaumgartnerPFA}). 
\end{proof}

The following definition and lemmas are formulations of well-known
techniques (originated in Shelah's use of elementary submodels in forcing)
and will simplify our further arguments.

\begin{defi}\label{defiWellReflects}
Let $P$ be a notion of forcing, $q\in P$ and let $\theta>(2^{|P|})^+$. Suppose $M\prec H(\theta)$ and 
$P, \pi_1, ...,\pi_k \in M$. We say that a formula $\phi(x_0, x_1, ..., x_k)$
{\rm well reflects $q$ in $(M;\pi_1, ...,\pi_k)$} whenever the following are 
satisfied:
\begin{enumerate}[i)]
\item $\phi(q,\pi_1,...,\pi_k)$ holds in $H(\theta)$;
\item whenever $s\in M$ is such that $\phi(s,\pi_1,...,\pi_k)$ holds in $M$, then
$q$ and $s$ are compatible.
\end{enumerate}
\end{defi}

\begin{defi}\label{defiSimplyFproper}
Suppose $\mathcal{F}\subseteq [\omega_2]^\omega$ and suppose $P$ is a notion 
of forcing. We say that $P$ is {\rm simply $\mathcal{F}$-proper} if there is $\theta$ 
such that whenever
\begin{enumerate}[a)]
\item $p\in P$,
\item $M\prec H(\theta)$, $M$ countable,
\item $p,\ P,\  \mathcal{F}\in M$,
\item $M\cap\omega_2\in \mathcal{F}$,
\end{enumerate}
then there is $p_0\leq p$ such that if $q\geq p_0$, then  
there are  $\pi_1, ...,\pi_k \in M$ and a formula $\phi(x_0, x_1,..., x_k)$
which {well reflects $q$ in $(M, \pi_1, ...,\pi_k)$}.
\end{defi}

\begin{lema}\label{lema6}
If $P$ is simply $\mathcal{F}$-proper, then $P$ is $\mathcal{F}$-proper.
\end{lema}
\begin{proof}
We will prove that whenever $M,\ p$ are as in a) - d)
of Definition \ref{defiSimplyFproper}, then $p_0$ is a $(P,M)$-generic condition.
Let $D\in M$ be dense, we will show that $D\cap M$ is predense below $p_0$. Let
$q\leq p_0$, we may w.l.o.g. assume that $q\in D$.
Let $\pi_1, ...,\pi_k \in M$  and $\phi(x_0, x_1,..., x_k)$
be such that $\phi(x_0, x_1,..., x_k)$ well  reflects $q$ in $(M, \pi_1, ...,\pi_k)$.
By i) of Definition \ref{defiWellReflects}, we have $\phi(q,\pi_1,...\pi_k)$ in 
$H(\theta)$. By its elementarity, $M$ satisfies the formula
``$\exists x\in P\  \phi(x,\pi_1,...\pi_k)\  \&\ x\in D$".
So let $s\in M$ witness this fact. Now by Definition \ref{defiWellReflects}.ii),
$s$ and $q$ are compatible, so $D\cap M$ contains a condition compatible with $q$ which 
proves that $D\cap M$ is predense below $q$ which completes the proof.
\end{proof}

\subsection{Adding a function with the strong property $\Delta$.}

Fix a family $\mathcal{F}\subseteq [\omega_2]^\omega$ satisfying 1) - 6) of Proposition \ref{prop1}.
We will assume familiarity of the reader with elementary submodels of structures
$H(\theta)$. In particular we will make use of facts such as that countable elements
of such models are their subsets or that such models contain $\omega$. See \cite{Dow2}
for more on this subject.
We consider the following forcing $P$ whose conditions
$p$ are of the form: $p=(a_p,f_p, A_p)$ where
\begin{enumerate}[a)]
\item $a_p\in [\omega_2]^{<\omega}$;
\item $f_p: [a_p]^2\rightarrow [\omega_2]^{<\omega}$;
\item $A_p\in[\mathcal{F}]^{<\omega}$;
\item $f_p(\alpha,\beta)\subseteq \bigcap\{X: X\in A_p, \ 
 {\alpha,\beta}\in X\}\cap\min\{\alpha,\beta\}$ for any distinct $\alpha,\beta\in a_p$.
\end{enumerate}

The order is just the inverse inclusion, i.e., $p\leq q$ if and only if
$a_p\supseteq a_q$, $f_p\supseteq f_q$, $A_p\supseteq A_q$.

\begin{fato}\label{fato7}
$P$ is simply $\mathcal{F}$-proper.
\end{fato}
\begin{proof}
Let $\theta=\omega_3$ and let $M$ and $p$ be as in a) - d) of Definition \ref{defiSimplyFproper}.
The existence of such an $M$ follows from the stationarity of $\mathcal{F}$. Let 
$X_0=M\cap\omega_2$. Let $p_0= (a_p, f_p, A_p\cup\{X_0\})$. Finally let $q\leq p_0$. The
proof consists of using Lemma \ref{lema6} and finding the parameters
$\pi_1, ...,\pi_k \in M$ and a formula $\phi(x_0, x_1,..., x_k)$
which well reflects $q$ in $(M, \pi_1, ...,\pi_k)$.

\vspace{6pt}

Define $q|M=(a_q\cap M, f_q|M, A_q\cap M)$. Introduce notation
$\delta=M\cap \omega_1=rank(M)$, where the second equality follows from 4) 
of Proposition \ref{prop1}. Note that $A_q\cap M=A_{q|M}= \{X\in A_q: X\subset X_0\}$. 
This follows from 5) of Proposition \ref{prop1}. The fact that $[M]^{<\omega}\subseteq M$
implies that $a_{q|M}, A_{q|M}\in M$. Also as d) of the definition
of the forcing is satisfied for $q$ and $\alpha, \beta \in a_q$, we have that 
$f_q(\alpha, \beta) \subseteq X_0 = M \cap \omega_2$ for $\alpha, \beta \in a_q \cap X_0$. 
So, we may conclude that
$f_{q|M} \in M$, in other words we have $q|M\in M\cap P$.
It is clear that $q|M\leq p$. By 6) of Proposition \ref{prop1} and the fact that 
$[M]^{<\omega}\subseteq M$, in $M$ there is a $Z\in\mathcal{F}$
such that $\bigcup\{X\cap M:\ rank(X)<\delta,\ X\in A_q\}\subseteq Z$.
Let $\phi(x_0, x_1, x_2, x_3, x_4)$ be the formula which says that
$x_0$ is a condition of the partial order $x_4$ which extends in $x_4$
the condition $x_3$ and such that the difference between
the first coordinate of $x_0$ and $x_2$ is disjoint from $x_1$.


\vspace{6pt}

\noindent \textbf{Claim.} \textit{$\phi(x_0, x_1, x_2, x_3, x_4)$ well-reflects
$q$ in $(M, Z, a_{q|M}, q|M, P )$.}

\vspace{6pt}

\noindent \textit{Proof of the Claim.}
It is clear that $\phi(q, Z, a_{q|M},q|M, P)$ holds in $H(\omega_3)$.
Now let $s\in M$ be a condition satisfying $\phi(s,Z,a_{q|M},q|M, P )$
i.e., $s$ extends in $P$ the condition $q|M$ and $a_s\setminus a_{q|M}$
is disjoint from $Z$. Define the common extension $r$ of $q$ and $s$
as follows: $a_r=a_s\cup a_q$, $f_r=f_s\cup f_q\cup h$,
$A_r=A_s\cup A_q$, where $h(\{\alpha,\beta\})=\emptyset$ for $\{\alpha,\beta\}
\in [a_s\cup a_q]^2-([a_s]^2\cup [a_q]^2)$. 
Such an $f_r$ is a function on $[a_r]^2$ since $q|M\geq q,s$.
Clearly all clauses of the definition of the forcing $P$ but d) are 
trivially satisfied by $r$. So let us prove d). Let $\alpha,\beta\in a_r$
and $X\in A_r$, we will consider a few cases.

\vspace{0.3cm}

\noindent \textit{Case 1.} $\alpha,\beta\in a_s$,
$X\in A_s$ 

It is trivial because $s\in P$.

\vspace{0.3cm}

\noindent \textit{Case 2.} $\alpha,\beta\in a_q$, $X\in A_q$ 

It is trivial because $q\in P$.

\vspace{0.3cm}

\noindent \textit{Case 3.} $\alpha,\beta\in a_s$, $X\in A_q$. 

Since $\phi(s,Z,a_{q|M},q|M, P )$ holds in $M$ we have that
either $rank(X)\geq\delta=rank(M\cap \omega_2)=rank(X_0)$ in which case
d) is satisfied because $f_r(\{\alpha,\beta\})=
f_s(\{\alpha,\beta\})\subseteq X_0\cap \min\{\alpha,\beta\}\subseteq X\cap\min\{\alpha,\beta\}$ by
d) for $s$ and 3) of Proposition \ref{prop1} or
$rank(X)<\delta$ and then by the definition of $\phi$ and $Z$
we get that $\alpha,\beta\in a_s\cap a_q$, so we are again in Case 2.

\vspace{0.3cm}

\noindent \textit{Case 4.} $\alpha,\beta\in a_q$, $X\in A_s$.

This means that $\alpha,\beta\in M$, because $s\in M$,
i.e., $\alpha,\beta\in  a_s\cap a_q$ so we are again in Case 1.

\vspace{0.3cm}

\noindent \textit{Case 5.} $\alpha\in a_s\setminus a_q$ and $\beta \in a_q\setminus a_s$.

Then $h(\{\alpha,\beta\})=\emptyset$.

\vspace{0.3cm}

The proof of the claim completes the proof of Fact \ref{fato7}.
\end{proof}

\begin{defi}\label{defiSupp}
For $p\in P$, call the set $a_p\cup f[[a_p]^2]\cup\bigcup A_p$ the support of $p$ and denote it by $supp(p)$.
\end{defi}

\begin{defi}\label{defiIsomCond}
We say that two conditions $p, q$ of $P$ are isomorphic (via $\pi:supp(p)\rightarrow supp(q)$)
if $\pi:supp(p)\rightarrow supp(q)$ is an order preserving bijection constant on 
$supp(p)\cap supp(q)$ and 
\begin{enumerate}[i)]
\item $\pi[a_p]=a_q$;
\item $\{\pi[X]:X\in A_p\}=A_q$;
\item $f_q(\{\pi(\alpha),\pi(\beta)\})=\pi[f_p(\{\alpha,\beta\})]$ for all $\alpha,\beta\in a_p$.
\end{enumerate}
\end{defi}

\begin{lema}\label{lema11}
Suppose $p, q\in P$ are isomorphic via $\pi:supp(p)\rightarrow supp(q)$. Then they are compatible.
\end{lema}
\begin{proof}
Define the common
extension $r$ of $p$ and $q$  as follows: $a_r=a_p\cup a_q$, $f_r=f_p\cup f_q\cup h$,
$A_r=A_p\cup A_q$, where $h(\{\alpha,\beta\})=\emptyset$ for $\{\alpha,\beta\}\in
[a_p\cup a_q]^2-([a_p]^2\cup[a_q]^2)$.
 The only non-automatic condition of the definition of $P$ which needs to be checked 
is d). 

\vspace{0.3cm}

\noindent \textit{Case 1.} $\alpha,\beta\in a_r$.

If $X\in A_r$,  we are trivially done.
If $X\in A_q$ and $\alpha,\beta\in X$, then $\alpha,\beta\in supp(p)\cap supp(q)$,
hence $\alpha,\beta\in a_p\cap a_q$ and hence again use d) for $q$.

\vspace{0.3cm}

\noindent \textit{Case 2.} $\alpha,\beta\in a_q$.

Similar to the previous case.

\vspace{0.3cm}

\noindent \textit{Case 3.} $\alpha\in a_r\setminus a_q$, $\beta\in a_q \setminus a_r$.

In this case $h$ is empty.
\end{proof}
\begin{fato}
Assuming CH the
forcing $P$ is $\omega_2$-c.c.
Thus by Fact \ref{fato7}, Lemma \ref{lema6} and Fact \ref{fato3}, $P$ preserves cardinals.
\end{fato}
\begin{proof}
By the previous lemma the proof is a standard application of
the $\Delta$-system lemma to the sequence of supports $\{supp({p_\xi}):\xi<\omega_2\}$
of some conditions $p_\xi\in P$ under our cardinal arithmetic assumption. 
\end{proof}

\begin{teo}\label{teoPiotrPantanos}
In $V^P$ there is a function $f:[\omega_2]^2\rightarrow[\omega_2]^{<\omega}$
with the strong property $\Delta$.
\end{teo}
\begin{proof}
 Clearly, we claim that
$f=\bigcup\{f_p:p\in G\}$ defines such a function,
where $G$ is a $P$-generic over $V$. Let ${\dot f}$ be a name for it.
 
Fix a set $A=\{\dot{a}^\alpha: \alpha<\omega_1\}$ 
of $P$-names
for elements of
 an uncountable  $\Delta$-system of  $n$-tuples $\dot{a}=\{{\dot a_i}: i<n\}$
of elements of $\omega_2$ for which there are bijections $e$ as  in Definition \ref{strongDeltaProperty} 
(any uncountable $\Delta$-system  has an uncountable such a subsystem).
Fix a condition $p\in P$.
  
Take a model $M\prec H(\omega_3)$ such that $M\cap\omega_2=X_0\in\mathcal{F}$ and
$p\in P\cap M$; $\mathcal{F}\in M$ and $\{\dot{a}^\alpha:\alpha<\omega_1\}\in M$. We
will show that there are  $\alpha_1<\alpha_2<\omega_1$ and $r\leq p$
such that $r$ forces 1), 2), 3) of Definition \ref{strongDeltaProperty}
 for $ \dot{a}^{\alpha_1}$ and $\dot{a}^{\alpha_2}$.

First take a condition $p_0\leq p $ as in  Fact \ref{fato7}, i.e.,
$a_{p_0}=a_p$, $f_{p_0}=f_p$, $A_{p_0}=A_{p}\cup X_0$.
Take $q\leq p_0$ and $\alpha_1\in\omega_1$ such that
there is $b$ such that $b\setminus M\not=\emptyset$,
$q\forces\dot{a}^{\alpha_1}={\check b}$ and $b\subseteq a_q$.
This can be done as $\{\dot{a}^\alpha:\alpha<\omega_1\}$
is a sequence of names for an uncountable  $\Delta$-system of sets and $|M|=\omega$.
Proceed as in the proof of Fact \ref{fato7}, i.e., choose $Z$ and $\phi$ as in Fact \ref{fato7}.
So, we have $\phi(q,Z,a_{q|M},q|M, P )$ in $H(\omega_3)$ and so by
the elementarity of $M$, we can find an $s$ and $\alpha_2$ such that $\phi(s,Z,a_{q|M},q|M, P)$ 
holds in $M$ and moreover there is $a$ such that $a\setminus(b\cap M)\in
 [M\setminus Z]^{<\omega}$ and such that
$s\forces\dot{a}^{\alpha_2}={\check a}$ and $a\subseteq a_s$.
Now we will obtain another amalgamation $r$ of $s$ and $q$
which will force 1), 2) and 3) of Definition \ref{strongDeltaProperty}. 
Let $a_r=a_s\cup a_q$,
 $f_r=f_s\cup f_q\cup h$. For  $\xi\in a_s\setminus a_q$
and $\eta\in a_q\setminus a_s$:
$$h(\{\xi,\eta\})=[A\cup B\cup C]\cap D \leqno **)$$
where
$$A=a\cap\min\{\xi,\eta\}$$
$$B=  \bigcup\{f_s(\{\tau,\xi\}):
\tau\in a\cap b, \tau<\eta\}$$
$$C=\bigcup\{f_q(\{\tau,\eta\}):
\tau\in a\cap b, \tau<\xi\}$$
$$D=\min\{\xi,\eta\}\cap
 \bigcap\{X\in A_q: \xi,\eta\in X, rank(X)\geq\delta\}$$

First let us check that $r$ is a common extension of $q$ and $s$. The 
proof also follows the cases as in the Claim in the proof of Fact \ref{fato7}. All are checked in the same manner
except for Case 5 where one may assume that
$X\in A_q$ as $\beta\not\in M$. This time the inclusion in the set D guarantees that d) holds in Case 5.

Now we will check 1), 2) and 3) of Definition \ref{strongDeltaProperty} for $a,b$ as above
and $f_r$. This will be enough since $r\forces\dot{a}^{\alpha_1}={\check b}, 
\dot{a}^{\alpha_2}={\check a}$  and $r\forces f_r\subseteq {\dot f}$. 
Suppose $\xi\in a\setminus b$ and $\eta\in b\setminus a$.
By the form of the definition of $f_r(\{\xi,\eta\})=h(\{\xi,\eta\})$
it will be enough to prove that the sets $A$, $B$ and $C$ are actually included in $\min\{\xi,\eta\}\cap X$
for  any $X\in A_q$ such that $rank(X)\geq \delta$ and $\xi, \eta\in X$.
So, let $X\in A_q$ be any such element that $rank(X)\geq \delta$ and $\xi, \eta\in X$.

\vspace{0.3cm}

For 1) of Definition \ref{strongDeltaProperty}, note that since $X_0=M\cap \omega_2$ and $rank(X_0)=\delta$ we have that
$M\cap\min\{\xi,\eta\}$ is included in $X$ by 3) of Proposition \ref{prop1}. Hence, as $a\subseteq M$, we have
$a\cap\min\{\xi,\eta\}\subseteq \min\{\xi,\eta\}\cap X$, that is we obtain 1).

\vspace{0.3cm}

To get 2) of Definition \ref{strongDeltaProperty} assume that $\tau \in a \cap b$ and $\tau<\xi$, hence $\min\{\tau,\eta\}\leq \min\{\xi,\eta\}$.
Note again, by 3) of Proposition \ref{prop1}, that $M\cap\xi\subseteq X\cap\xi$ which implies in this case that
$\tau\in X$.
Hence, since $\tau,\eta\in a_q$, by d) of the definition of the forcing,
we have that $f_q(\{\tau,\eta\})\subseteq X\cap\min\{\tau,\eta\} \subseteq X\cap\min\{\xi,\eta\}$, so we obtain 2).

\vspace{0.3cm}

To get 3) of Definition \ref{strongDeltaProperty} assume that $\tau \in a \cap b$ and $\tau<\eta$,
hence $min\{\tau,\xi\}\leq\min\{\xi,\eta\}$. We have that $\tau,\xi\in M\cap\xi$,
and again $M\cap\xi\subseteq X$. 
Hence, since $\xi,\tau\in a_s$, by d) of the definition of the forcing,
and the fact that $s\in M$, we have that 
$f_s(\{\tau,\xi\})\cap\min\{\tau,\xi\}\subseteq X\cap\min\{\tau,\xi\} \subseteq X\cap\min\{\xi,\eta\}$, so we obtain 3).
\end{proof}

\begin{remark}
For any $k\leq\omega$ and any uncountable $\Delta$-system
one can have $k$ sets satisfying 1), 2) and 3) of Definition \ref{strongDeltaProperty}. 
This follows from the Dushnik-Miller theorem see \cite{Jech} Theorem 9.7.
\end{remark}

\subsection{CH and the strong property $\Delta$.}

In this section we prove that CH implies that
there is no function $f$ such as in the previous section, even if we allow
$f$ to take countable sets as values. This also proves that
the strong property $\Delta$ cannot be obtained as in Baumgartner and Shelah \cite{BaumgartnerShelah},
that is, by a forcing which preserves CH.

\begin{prop}
(CH) There is no $f:[\omega_2]^2\rightarrow[\omega_2]^\omega$ such that
for every $\Delta$-system $\mathcal{A}$ of finite subsets of $\omega_2$
of cardinality $\aleph_1$, there exist distinct $a,b\in\mathcal{A}$ such that
$$\forall\xi\in a\setminus b\forall \eta\in b\setminus a\ \ \ a\cap\xi\cap\eta\subseteq f(\{\xi,\eta\}).\leqno *)$$
\end{prop}
\begin{proof}
Suppose that $f:[\omega_2]^2\rightarrow[\omega_2]^\omega$.
For an $A\subseteq\omega_2$ and $\xi\in\omega_2\setminus A$ define
$$f_{A,\xi}:A\rightarrow[A]^\omega,\ \ \ f_{A,\xi}(\eta)=f(\{\eta,\xi\})\cap A\ \ \ \forall\eta\in A.$$
Let $M\prec H(\omega_3)$ be closed under its countable subsets (here we use CH)
$|M|=\omega_1$, $\omega_1\subseteq M$; $\omega_1,\omega_2, f\in M$ and such that 
$\sup(M\cap\omega_2)=\gamma$ has an uncountable cofinality.

\vspace{0.3cm}

By recursion construct a sequence $(\alpha_\xi,\beta_\xi)_{\xi<\omega_1}$ which satisfies:
\begin{enumerate}[1)]
\item $\omega_1<\alpha_\xi,\beta_\xi\in M\cap\gamma$;
\item $\alpha_\xi<\beta_\xi<\alpha_\eta$ for all $\xi<\eta$;
\item $f_{A_\eta,\beta_\eta}=f_{A_\eta,\gamma}$ where $A_\eta=\{\alpha_\xi,\beta_\xi:\xi<\eta\}$;
\item $\alpha_\eta\not\in f(\{\beta_\eta,\gamma\})$.
\end{enumerate}

To justify that this construction can be carried out assume that we have
$A_\eta$ satisfying  1)-4) and let us show how to obtain
$\alpha_\eta,\beta_\eta$. As $A_\eta\subseteq M$ and  $M$ is closed under its countable sets
we have $A_\eta\in M$. Also $f_{A_\eta,\gamma}\in M$ as $M$ is closed under countable sets.
Hence, by the elementarity there is $\beta_\eta\in M\setminus\sup(A_\eta)$ such that
$f_{A_\eta,\beta_\eta}=f_{A_\eta,\gamma}$  and  $cf(\beta_\eta)=\omega_1$.
Now $f(\{\beta_\eta,\gamma\})\cap M$ is in $M$ again, so using the fact that
$cf(\beta_\eta)=\omega_1$ we can find $\alpha_\eta\in M$ satisfying 
$A_\eta<\alpha_\eta<\beta_\eta$ and $\alpha_\eta\not\in f(\{\beta_\eta,\gamma\})$
which completes the construction.

\vspace{0.3cm}

Now define $\gamma_\eta<\omega_1$ such that for $\xi<\eta<\omega_1$ we have
$$\gamma_\xi<\gamma_\eta\not\in\bigcup\{f(\{\beta_\xi,\beta_\eta\}):\xi<\eta\}.$$
Finally define  $\mathcal{A}=\{\{\gamma_\xi,\alpha_\xi,\beta_\xi\}:\xi<\omega_1\}$.
Suppose  $\xi<\eta$. Note that, as by 4),
 $\alpha_\xi\not\in f(\{\beta_\xi,\gamma\})$ and by 3), $f(\{\beta_\xi,\gamma\})=f(\{\beta_\xi,\beta_\eta\})$, 
we have 
$$\alpha_\xi\in (\beta_\xi\cap\beta_\eta)\setminus f(\{\beta_\xi,\beta_\eta\}).$$
But on the other hand, by the definition of $\gamma_\eta$, we have
$$\gamma_\eta\in (\beta_\xi\cap\beta_\eta)\setminus f(\{\beta_\xi,\beta_\eta\}),$$
which shows that the inclusion *) of the proposition holds for no $a,b\in\mathcal{A}$.
\end{proof}


\bibliographystyle{amsplain}

\begin{thebibliography}{30}

\bibitem{BaumgartnerPFA}
J.~E. Baumgartner, \emph{Applications of the proper forcing axiom}, Handbook of
  set-theoretic topology, North-Holland, Amsterdam, 1984, pp.~913--959.

\bibitem{BaumgartnerShelah}
J.~E. Baumgartner and S.~Shelah, \emph{Remarks on superatomic {B}oolean
  algebras}, Ann. Pure Appl. Logic \textbf{33} (1987), no.~2, 109--129.

\bibitem{BorweinVanderwerff}
J.~M. Borwein and J.~D. Vanderwerff, \emph{Banach spaces that admit support
  sets}, Proc. Amer. Math. Soc. \textbf{124} (1996), no.~3, 751--755.

\bibitem{DevilleGodefroyZizler}
R.~Deville, G.~Godefroy, and V.~Zizler, \emph{Smoothness and renormings in
  {B}anach spaces}, Pitman Monographs and Surveys in Pure and Applied
  Mathematics, vol.~64, Longman Scientific \& Technical, Harlow, 1993.

\bibitem{Dow2}
A.~Dow, \emph{An introduction to applications of elementary submodels to
  topology}, Topology Proc. \textbf{13} (1988), no.~1, 17--72.

\bibitem{BookBiorthogonal}
P.~H{\'a}jek, V.~Montesinos~Santaluc{\'{\i}}a, J.~Vanderwerff, and V.~Zizler,
  \emph{Biorthogonal systems in {B}anach spaces}, CMS Books in
  Mathematics/Ouvrages de Math\'ematiques de la SMC, 26, Springer, New York,
  2008.

\bibitem{HaydonCSQ}
R.~Haydon, \emph{A counterexample to several questions about scattered compact
  spaces}, Bull. London Math. Soc. \textbf{22} (1990), no.~3, 261--268.

\bibitem{HodelHBST}
R.~Hodel, \emph{Cardinal functions. {I}}, Handbook of set-theoretic topology,
  North-Holland, Amsterdam, 1984, pp.~1--61.

\bibitem{Jech}
T.~Jech, \emph{Set theory}, Springer Monographs in Mathematics,
  Springer-Verlag, Berlin, 2003, The third millennium edition, revised and
  expanded.

\bibitem{JimenezMoreno}
M.~Jim{\'e}nez~Sevilla and J.-P. Moreno, \emph{Renorming {B}anach spaces with
  the {M}azur intersection property}, J. Funct. Anal. \textbf{144} (1997),
  no.~2, 486--504.

\bibitem{JuhaszSoukup}
I.~Juh{\'a}sz and L.~Soukup, \emph{How to force a countably tight, initially
  {$\omega\sb 1$}-compact and noncompact space?}, Topology Appl. \textbf{69}
  (1996), no.~3, 227--250.

\bibitem{JuhaszWeiss}
I.~Juh{\'a}sz and W.~Weiss, \emph{On thin-tall scattered spaces}, Colloq. Math.
  \textbf{40} (1978/79), no.~1, 63--68.

\bibitem{Just}
W.~Just, \emph{Two consistency results concerning thin-tall {B}oolean
  algebras}, Algebra Universalis \textbf{20} (1985), no.~2, 135--142.

\bibitem{KoszmiderSemimorasses}
P.~Koszmider, \emph{Semimorasses and nonreflection at singular cardinals}, Ann.
  Pure Appl. Logic \textbf{72} (1995), no.~1, 1--23.

\bibitem{KoszmiderSCUF}
\bysame, \emph{On strong chains of uncountable functions}, Israel J. Math.
  \textbf{118} (2000), 289--315.

\bibitem{KoszmiderUMSUF}
\bysame, \emph{Universal matrices and strongly unbounded functions}, Math. Res.
  Lett. \textbf{9} (2002), no.~4, 549--566.

\bibitem{MartinezIJM}
J.~C. Mart{\'{\i}}nez, \emph{A consistency result on thin-very tall {B}oolean
  algebras}, Israel J. Math. \textbf{123} (2001), 273--284.

\bibitem{Mazur}
S.~Mazur, \emph{\"Uber schwache {K}onvergenz in den {R}aumen {$L^p$}}, Studia
  Math. \textbf{4} (1933), 129--133.

\bibitem{NamiokaPhelps}
I.~Namioka and R.~R. Phelps, \emph{Banach spaces which are {A}splund spaces},
  Duke Math. J. \textbf{42} (1975), no.~4, 735--750.

\bibitem{NegrepontisHBST}
S.~Negrepontis, \emph{Banach spaces and topology}, Handbook of set-theoretic
  topology, North-Holland, Amsterdam, 1984, pp.~1045--1142.

\bibitem{Ostaszewski}
A.~J. Ostaszewski, \emph{A countably compact, first-countable, hereditarily
  separable regular space which is not completely regular}, Bull. Acad. Polon.
  Sci. S\'er. Sci. Math. Astronom. Phys. \textbf{23} (1975), no.~4, 431--435.

\bibitem{Rabus}
M.~Rabus, \emph{An {$\omega\sb 2$}-minimal {B}oolean algebra}, Trans. Amer.
  Math. Soc. \textbf{348} (1996), no.~8, 3235--3244.

\bibitem{Rajagopalan}
M.~Rajagopalan, \emph{A chain compact space which is not strongly scattered},
  Israel J. Math. \textbf{23} (1976), no.~2, 117--125.

\bibitem{Roitman}
J.~Roitman, \emph{Introduction to modern set theory}, Pure and Applied
  Mathematics (New York), John Wiley \& Sons Inc., New York, 1990, A
  Wiley-Interscience Publication.

\bibitem{Shapirovskii}
B.~{\`E}. Shapirovski{\u\i}, \emph{Cardinal invariants in compacta}, Seminar on
  General Topology, Moskov. Gos. Univ., Moscow, 1981, pp.~162--187.

\bibitem{ShelahProperForcing}
S.~Shelah, \emph{Proper forcing}, Lecture Notes in Mathematics, vol. 940,
  Springer-Verlag, Berlin, 1982.

\bibitem{TodorcevicProper}
S.~Todorcevic, \emph{A note on the proper forcing axiom}, Axiomatic set theory
  ({B}oulder, {C}olo., 1983), Contemp. Math., vol.~31, Amer. Math. Soc.,
  Providence, RI, 1984, pp.~209--218.
 
\bibitem{TodorcevicDirectedSets}
\bysame, \emph{Directed sets and cofinal types}, Trans. Amer. Math. Soc.
  \textbf{290} (1985), no.~2, 711--723.

\bibitem{TodorcevicBiorthogonal}
\bysame, \emph{Biorthogonal systems and quotient spaces via {B}aire category
  methods}, Math. Ann. \textbf{335} (2006), no.~3, 687--715.

\bibitem{Velleman2}
D.~Velleman, \emph{Simplified morasses}, J. Symbolic Logic \textbf{49} (1984),
  no.~1, 257--271.

\bibitem{Zwicker}
W.~S. Zwicker, \emph{{$P\sb k\lambda $} combinatorics. {I}. {S}tationary coding
  sets rationalize the club filter}, Axiomatic set theory (Boulder, Colo.,
  1983), Contemp. Math., vol.~31, Amer. Math. Soc., Providence, RI, 1984,
  pp.~243--259.

\end{thebibliography}

\end{document}